\documentclass[12pt]{amsart}
\usepackage{amssymb}

\usepackage{amsfonts}
\usepackage{amssymb,latexsym}
\usepackage{enumerate}
\usepackage{mathrsfs}
\usepackage{url} \usepackage{color}\makeatletter
\@namedef{subjclassname@2010}{%
  \textup{2010} Mathematics Subject Classification}
\makeatother

  [2002/01/19 v2.2g %
    AMS font definitions%
  ]
\DeclareFontFamily{U}{euf}{}
\DeclareFontShape{U}{euf}{m}{n}{%
  <5><6><7><8><9>gen*eufm%
  <10><10.95><12><14.4><17.28><20.74><24.88>eufm10%
  }{}
\DeclareFontShape{U}{euf}{b}{n}{%
  <5><6><7><8><9>gen*eufb%
  <10><10.95><12><14.4><17.28><20.74><24.88>eufb10%
  }{}

  [2002/01/19 v2.2g %
    AMS font definitions%
  ]
\DeclareFontFamily{U}{msb}{}
\DeclareFontShape{U}{msb}{m}{n}{%
  <5><6><7><8><9>gen*msbm%
  <10><10.95><12><14.4><17.28><20.74><24.88>msbm10%
  }{}

  [2002/01/19 v2.2g %
    AMS font definitions%
  ]
\DeclareFontFamily{U}{msa}{}
\DeclareFontShape{U}{msa}{m}{n}{%
  <5><6><7><8><9>gen*msam%
  <10><10.95><12><14.4><17.28><20.74><24.88>msam10%
  }{}

\newtheorem{theorem}{Theorem}[section]
\newtheorem{lemma}[theorem]{Lemma}
\newtheorem{proposition}[theorem]{Proposition}
\newtheorem{corollary}[theorem]{Corollary}

\theoremstyle{definition}

\numberwithin{equation}{section}

\numberwithin{equation}{section} 

\begin{document}

\title[]
{An analogue of Kummer's relation
between the ideal class number and the unit index of cyclotomic fields}

\author{Su Hu}
\address{Department of Mathematics, South China University of Technology, Guangzhou 510640, China}
\address{State Key Laboratory of Cryptology, P. O. Box 5159, Beijing, 100878, China}
\email{mahusu@scut.edu.cn}

\author{Min-Soo Kim}
\address{Division of Mathematics, Science, and Computers, Kyungnam University, 7(Woryeong-dong) kyungnamdaehak-ro, Masanhappo-gu, Changwon-si,
Gyeongsangnam-do 51767, Republic of Korea}
\email{mskim@kyungnam.ac.kr}

\author{Yan Li*}
\address{Department of Applied Mathematics, China Agricultural university, Beijing 100083, China}
\address{State Key Laboratory of Cryptology, P. O. Box 5159, Beijing, 100878, China}
\email{liyan\_00@cau.edu.cn }

\thanks{*Corresponding author}

\begin{abstract}
In this paper, we obtain a formula for the special value of Euler-Dirichlet $L$-function $L_E(s,\chi)$ at $s=1$. This leads to another class number formula of $\mathbb{Q}(\mu_{m})^{+}$, the maximal real
subfield of $m$th cyclotomic field. From this formula, we construct a new type of cyclotomic units in  $\mathbb{Q}(\mu_{p^{n}})$,  which implies a similar Kummer's relation
between the ideal class number of $\mathbb{Q}(\mu_{p^{n}})^{+}$ and the unit index.\end{abstract}

\maketitle
\section{Introduction}
Throughout the paper, $\chi$ will always be a primitive Dirichlet character with conductor $f_{\chi}$ or $f$. Therefore, for $n\in \mathbb{Z}$, $\chi(n)=0\Leftrightarrow\textrm{gcd}(n,f)>1$.

An alternating form of  Dirichlet $L$-function:
\begin{equation}\label{E-zeta0}
L_E(s,\chi)=\sum_{n=1}^\infty\frac{(-1)^{n-1}\chi(n)}{n^s}, ~~\textrm{Re}(s)>0,
\end{equation}
is called Euler-Dirichlet $L$-function.
We see that $L_{E}(s,\chi)$ (\ref{E-zeta0})  is indeed the following Dirichlet eta function with a character:
\begin{equation}\label{Deta}
\eta(s)=\sum_{n=1}^{\infty}\frac{(-1)^{n-1}}{n^{s}},~~~~\textrm{Re}(s)>0.
\end{equation}
The above $\eta(s)$ is a particular case of Witten's zeta functions in mathematical physics (see \cite[p. 248, (3.14)]{Min}) and it has been used by Euler to obtain a functional equation of Riemann zeta function $\zeta(s)$ (see \cite[p.273--276]{Weil}).
There is also a connection between $L_{E}(s,\chi)$ and the ideal class group of the
$p^{n+1}$-th cyclotomic field where $p$ is a prime number.
For details, we refer to a recent paper~\cite{HK-I}, especially~\cite[Propositions 3.2 and 3.4]{HK-I}.

 In this paper, we obtain a formula for the special value of  Euler-Dirichlet $L$-function $L_E(s,\chi)$ at $s=1$ (see Theorem \ref{main1}). As a consequence, we get another class number formula of $\mathbb{Q}(\mu_{m})^{+}$, the maximal real
subfield of $m$th cyclotomic field  (see Theorem \ref{main2}).
From this formula, we construct a new type of cyclotomic units in  $\mathbb{Q}(\mu_{p^{n}})$ (see Eq. (\ref{cycunits})), which implies a similar Kummer's relation
between the class number of $\mathbb{Q}(\mu_{p^{n}})^{+}$ and the unit index (see Theorem  \ref{Cor2}).
\subsection{Background}
The Dirichlet $L$-function $L(s,\chi)$ ($L$-series)  is  defined by
\begin{equation}\label{D-zeta}
L(s,\chi)=\sum_{n=1}^\infty\frac{\chi(n)}{n^s},
\end{equation}
where the series on the right is absolutely convergent for Re$(s)>1$ and is conditionally convergent for Re$(s)>0$ for non-principal $\chi$
(see \cite{HKT}).
Let \begin{equation}\label{ex-re}
\zeta_{f}=e^{\frac{2\pi i}f}=\cos\left(\frac{2\pi}{f}\right)+i\sin\left(\frac{2\pi}f\right)
\end{equation}
be a primitive $f$th root of unity and
\begin{equation}\label{Gauss}
\tau(\chi)=\sum_{r~(\text{mod}~f)}\chi(r)\zeta_{f}^r,
\end{equation}
be the Gauss sum corresponding to the character $\chi$,
where $r$ runs through a full (or a reduced) system of residues modulo $f.$
If $\chi$ is non-principal, the following result on the special value of $L(s,\chi)$ at $s=1$ is well-known.
\begin{theorem}[See Washington  {\cite[p.38, Theorem 4.9]{Washington}}]\label{C1}
\begin{equation}\label{L1}
\begin{aligned} L(1,\chi)&=-\frac{2\tau(\chi)}{f}\sum_{1\leq k<f/2}\overline{\chi}(k)\log \sin\left(\frac{k\pi}{f}\right)
\\&=-\frac{\tau(\chi)}f \sum_{k=1}^{f} \overline{\chi}(k)\log |1-\zeta_{f}^k|,
~~&\textrm{if}~\chi~\textrm{ is even}, ~\textrm{i.e.}~\chi(-1)=1.\\
 L(1,\chi)&= \frac{\pi i \tau(\chi)}{f^{2}}\sum_{k=1}^{f} \overline{\chi}(k)k,
~~&\textrm{if}~\chi~\textrm{ is odd}, ~\textrm{i.e.}~\chi(-1)=-1.
\end{aligned}
\end{equation}
\end{theorem} This formula has an important application in algebraic number theory
by connecting $L(1,\chi)$ with the class number formula of abelian fields $K/\mathbb{Q}$, that is:
\begin{theorem}[See Lang  {\cite[p.77]{Lang}}]\label{C2}
\begin{equation}\label{CNF}
\frac{2^{r_{1}}(2\pi)^{r_{2}}hR}{wd^{1/2}}=\prod_{\chi\neq 1}L(1,\chi),
\end{equation}
where the product is taken over all the primitive characters induced by the characters of Gal($K/\mathbb{Q}$)
and
$$\begin{aligned}
w&=w_{K}=\textrm{number~~ of~~ roots~~ of~~ unity~~ in~~} K,\\
h&=h_{K}=\textrm{class~~ number~~ of~~} K,\\
R&=R_{K}=\textrm{regulator~~ of~~} K,\\
d&=d_{K}=\textrm{absolute~~ value~~ of~~ the~~ discriminant.}
\end{aligned}$$
If $K$ is real, then $r_{1}=[K:\mathbb{Q}]$ and $r_2=0$; if $K$ is complex, then $r_{1}=0$ and $r_{2}=\frac{1}{2}[K:\mathbb{Q}]$.
\end{theorem}
Assume that $m$ is odd or $m\equiv 0~(\textrm{mod}~4)$, $K=\mathbb{Q}(\mu_{m})$ and $K^{+}=\mathbb{Q}(\mu_{m})^{+}$
be the $m$-th cyclotomic field and its maximal real subfield, respectively; $h$ and $h^{+}$ be the class number of $K$ and $K^{+}$, respectively.

For convenience of the notations, we also denote $\mathbb{Z}/m\mathbb{Z}$ by $\mathbb{Z}(m)$.
Let $f_{\chi}$ be the conductor of $\chi$. Introduce the group
\begin{equation}
\begin{aligned}
G=\mathbb{Z}(m)^{*}/\pm 1~\textrm{and}~G_{\chi}=\mathbb{Z}(f_{{\chi}})^{*}/\pm 1
\end{aligned}
\end{equation}
Combining (\ref{L1}) and (\ref{CNF}), we have the class number formula of $K^{+}$.
\begin{theorem}[See Lang  {\cite[p.81]{Lang}}]\label{C3}
\begin{equation}\label{CNF+}
h^{+}=\frac{1}{R^{+}}\prod_{\chi\neq 1}\sum_{k\in G_{\chi}}-\chi(k)\log |1-\zeta_{f_{\chi}}^{k}|,
\end{equation}
where the product over $\chi\neq 1$ is taken over the non-trivial characters of $G$, or equivalently, the non-trivial even characters of $\mathbb{Z}(m)^{*}$.
\end{theorem}
For $k$ prime to $m$, we let
\begin{equation*}
g_{k}=\frac{\zeta_{m}^{k}-1}{\zeta_{m}-1}.
\end{equation*}
Then $g_{k}$ is called a cyclotomic unit. It is easy to see that $g_{k}$ is equal to a real unit times a root of unity. Since $\zeta_{m}^{k}$ only depends  on the residue class of $k$ mod $m$, without loss
of generality, we may assume that $k$ is odd.
Then
\begin{equation*}
\zeta_{m}^{-v}g_{k}~~\textrm{for}~v=\frac{k-1}{2}
\end{equation*}
is real (i.e. fixed under $\sigma_{-1}$), and call it the real cyclotomic unit.

Let $E$ be the group of unit in $K$ and $\mathcal{E}$ be the the subgroup of $E$ generated by the roots of unity and the cyclotomic units.
Let $E^{+}$ be the group of unit in $K^{+}$ and $\mathcal{E}^{+}$ be the the subgroup of $E^{+}$ generated by $\pm 1$ and real cyclotomic units.
Assume $m=p^{n}$ is a prime power, we have
\begin{equation*}
E/\mathcal{E}\cong E^{+}/\mathcal{E}^{+}.
\end{equation*}
 From (\ref{CNF+}) and the Dedekind determinant formula (see Theorem \ref{DF}), we have the following Kummer's result
related to the class number $h^{+}$ and the unit indexes  $(E:\mathcal{E})=(E^{+}:\mathcal{E}^{+})$.
\begin{theorem}[See Lang  {\cite[p.85, Theorem 5.1]{Lang}}]\label{C4}
 Let $K=\mathbb{Q}(\mu_{m})$, $K^{+}=\mathbb{Q}(\mu_{m})^{+}$ and  $h^{+}$ be the class number of $K^{+}$. Assume $m=p^{n}$ is a prime power. Then
\begin{equation*}
h^{+}=(E:\mathcal{E})=(E^{+}:\mathcal{E}^{+}).
\end{equation*}
\end{theorem}
\subsection{Our results}
First, we will calculate the special value of $L_E(s,\chi)$ at $s=1$ (comparing with Theorem~\ref{C1} above).
\begin{theorem}\label{main1}
\begin{equation}\label{LE1}
\begin{aligned}
L_E(1,\chi)&=\frac{2\tau(\chi)}{f}\sum_{1\leq k<f/2} \overline{\chi}(k)\log\left|\cos\left(\frac{\pi k}{f}\right)\right|     \\
&=\frac{\tau(\chi)}f \sum_{k=1}^{f} \overline{\chi}(k)\log |1+\zeta_{f}^k|,\\&\textrm{if}~\chi~\textrm{ is even}, ~\textrm{i.e.}~\chi(-1)=1.\\
L_E(1,\chi)&=-\frac{2\pi i \tau(\chi)}{f^2}\sum_{1\leq k < f/2}\overline{\chi}(k)k,\\&\textrm{if}~\chi~\textrm{ is odd}, ~\textrm{i.e.}~\chi(-1)=-1.
 \end{aligned}
 \end{equation}
 \end{theorem}

Assume that $m$ is odd or $m\equiv 0~(\textrm{mod}~4)$, $K=\mathbb{Q}(\mu_{m})$ and $K^{+}=\mathbb{Q}(\mu_{m})^{+}$
be the $m$-th cyclotomic field and its maximal real subfield, respectively; $h$ and $h^{+}$ be the class number of $K$ and $K^{+}$, respectively.
Denote by
\begin{equation}\label{etadef}\eta=\prod_{\chi\neq 1}(1-\chi(2))\end{equation}
where the product over $\chi\neq 1$ is taken over the non-trivial characters of $G$, or equivalently, the non-trivial even characters  $\mathbb{Z}(m)^{*}$.
Combing (\ref{LE1}) and the class number formula of abelian fields (see (\ref{CNF})), we also have another class
number formula of $K^{+}$ (comparing with Theorem \ref{C3} above).
\begin{theorem}\label{main2}
\begin{equation}\label{CNF+A}
\eta h^{+}=\frac{1}{R^{+}}\prod_{\chi\neq 1}\sum_{k\in G_{\chi}}\chi(k)\log |1+\zeta_{f_{\chi}}^{k}|,
\end{equation}
where the product over $\chi\neq 1$ is taken over the non-trivial characters of $G$, or equivalently, the non-trivial even characters  $\mathbb{Z}(m)^{*}$.
\end{theorem}\label{units}
For $k$ prime to $m$, we define a new type of cyclotomic units to be
\begin{equation}\label{cycunits}
\tilde{g}_{k}=\frac{\zeta_{m}^{k}+1}{\zeta_{m}+1}
\end{equation}
(comparing with the definition of $g_{k}$ above). Then $\tilde{g}_{k}$ is equal to a real unit times a root of unity. Since $\zeta_{m}^{k}$ only depends  on the residue class of $k$ mod $m$, without loss
of generality, we may assume that $k$ is odd.
Then \begin{equation*}
\zeta_{m}^{-v}\tilde{g}_{k}~~\textrm{for}~v=\frac{k-1}{2}
\end{equation*}
is   real (i.e. fixed under $\sigma_{-1}$), and we call it a new type of cyclotomic units.

Let $E$ be the group of unit in $K$ and $\tilde{\mathcal{E}}$ be the subgroup of $E$ generated by the roots of unity and the new type cyclotomic units defined above.
Let $E^{+}$ be the group of unit in $K^{+}$ and $\tilde{\mathcal{E}}^{+}$ be the the subgroup of $E^{+}$ generated by $\pm 1$ and the new type real cyclotomic units introduced above.
We have the following isomorphism.
 \begin{proposition}\label{Proposition1.7}
 Assume $m=p^{n}$ is a prime power, we have   \begin{equation*}
E/\tilde{\mathcal{E}}\cong E^{+}/\tilde{\mathcal{E}}^{+}.
\end{equation*}
\end{proposition}

From (\ref{CNF+A}) and the Dedekind determinant formula (see Theorem \ref{DF}), we also have a new formula related to $h^{+}$ and the unit indexes $(E:\tilde{\mathcal{E}})$ and $(E^{+}:\tilde{\mathcal{E}}^{+})$ (comparing with Theorem \ref{C4} above).
\begin{theorem}\label{Cor2} Let $K=\mathbb{Q}(\mu_{m})$, $K^{+}=\mathbb{Q}(\mu_{m})^{+}$ and  $h^{+}$ be the class number of $K^{+}$. Assume $m$ is a prime power. If either $m=2^{n}$; or $m=p^{n}$ is an odd prime power and $-1, 2$ generate the group $\mathbb{Z}(m)^{*}$, then $\eta\neq0$ and
\begin{equation*}
|\eta| h^{+}=(E:\tilde{\mathcal{E}})=(E^{+}: \tilde{\mathcal{E}}^{+}).
\end{equation*}
 Otherwise $\eta=0$ and $\mathcal{E}$ (resp. $\mathcal{E}^{+}$) is of infinite index in $E$ (resp. $E^{+}$).
\end{theorem}

\section{Proof of Theorem \ref{main1}}
Let $\chi$ be a non-trivial Dirichlet  character with conductor $f$.
 We rearrange the terms in the series for $L_E(s,\chi)$ according to the residue classes mod $f.$
That is, we write
$$n=qf+r, \quad\text{where $1\leq r\leq f$ and $q=0,1,2,\ldots,$}$$
and obtain
\begin{equation}\label{ca-1}
\begin{aligned}
-L_E(s,\chi)&=\sum_{r=1}^f\chi(r)\sum_{q=0}^\infty\frac{(-1)^{qf+r}}{(qf+r)^s} \\
&=\sum_{(r,f)=1}\chi(r)\sum_{n\equiv r~(\text{mod}~f)}\frac{(-1)^{n}}{n^s},
\end{aligned}
\end{equation}
since $\chi(r)=0$ if $(r,f)>1.$
The inner series can be written in the form
\begin{equation}\label{cn}
\sum_{n=1}^\infty\frac{(-1)^nc_n}{n^s},
\end{equation}
where
$$c_n=\begin{cases}
1 &\text{for } n\equiv r \pmod f \\
0 &\text{for } n\not\equiv r \pmod f.
\end{cases}$$
To find a convenient way of writing the coefficients $c_n,$ we consider the following formula:
$$\sum_{k=0}^{f-1}\zeta_{f}^{\ell k}=\begin{cases}
f &\text{for } \ell\equiv 0 \pmod{f} \\
0 &\text{for } \ell\not\equiv 0 \pmod f,
\end{cases}$$
where
\begin{equation}\label{ex-re}
\zeta_{f}=\cos\left(\frac{2\pi}{f}\right)+i\sin\left(\frac{2\pi}f\right)=e^{\frac{2\pi i}f}
\end{equation}
is a primitive $f$th root of unity. We remark
\begin{equation}\label{cn-ex}
c_n=\frac1f\sum_{k=0}^{f-1}\zeta_{f}^{(r-n)k}
\end{equation}
(see \cite[p.~332]{BS}).
Therefore, combing with (\ref{ca-1}), (\ref{cn}) and (\ref{cn-ex}) we have the identity
\begin{equation}\label{ca-2}
\begin{aligned}
-L_E(s,\chi)
&=\sum_{(r,f)=1}\chi(r)\sum_{n=1}^\infty\frac1f \sum_{k=0}^{f-1}\zeta_{f}^{(r-n)k} \frac{(-1)^{n}}{n^s} \\
&=\frac1f \sum_{k=1}^{f-1}\left(\sum_{(r,f)=1}\chi(r)\zeta_{f}^{rk} \right)\sum_{n=1}^\infty\frac{(-1)^{n}\zeta_{f}^{-nk} }{n^s},
\end{aligned}
\end{equation}
since $\sum_{(r,f)=1}\chi(r)=0$ if $\chi$ is not the principal character.

The Gauss sum $\tau(\chi)$ (\ref{Gauss}) depends not only on $\chi$, but also on the choice of the root $\zeta_{f}.$
Throughout the proof of this theorem, we  always assume   $\zeta_{f}$ is $\cos(2\pi/f)+i\sin(2\pi/f).$

We see that
\cite[Lemma 4.7]{Washington}
\begin{equation}\label{D-Bu}
\sum_{r~(\text{mod}~f)}\chi(r)\zeta_{f}^{rk}=\overline{\chi}(k)\tau(\chi).
\end{equation}
Using (\ref{D-Bu}) we can write (\ref{ca-2}) in the form
\begin{equation}\label{conv}
\begin{aligned}
-L_E(s,\chi)
&=\frac{\tau(\chi)}f  \sum_{k=1}^{f-1} \overline{\chi}(k)\sum_{n=1}^\infty\frac{(-1)^{n}\zeta_{f}^{-nk} }{n^s}.
\end{aligned}
\end{equation}
The above series (\ref{conv}) converges for $0<s<\infty$ and represents a continuous function of $s.$
Hence we may set $s=1$ in this last equation and obtain
\begin{equation}\label{le=1}
-L_E(1,\chi)=\frac{\tau(\chi)}f \sum_{k=1}^{f-1} \overline{\chi}(k)\sum_{n=1}^\infty\frac{(-1)^{n}\zeta_{f}^{-nk} }{n}.
\end{equation}
To find the sum of the inner series on the right in (\ref{le=1}), we consider
\begin{equation}\label{ln-def}
\log (1+z)=\sum_{n=1}^\infty\frac{(-1)^{n-1}}{n}z^n,\quad\text{convergent for }|z|<1.
\end{equation}
It is well known that
the analytic function $\log (1+z)$ defined by (\ref{ln-def}) has $z=-1$ as its only singular point at finite distance.
Since this series also converges at the point
$z=\zeta_{f}^{-k}$ (on the unit circle),
then by Abel's theorem, we have
\begin{equation}
\sum_{n=1}^\infty\frac{(-1)^{n}\zeta_{f}^{-nk} }{n}=-\log (1+\zeta_{f}^{-k})
\end{equation}
and hence
\begin{equation}\label{sim Le=1}
L_E(1,\chi)=\frac{\tau(\chi)}f \sum_{k=1}^{f} \overline{\chi}(k)\log (1+\zeta_{f}^{-k}),
\end{equation}
thus  we have obtained a finite expression for the series $L_E(1,\chi).$

The formula (\ref{sim Le=1}) can be further investigated and
considerably simplified as follows. Let \begin{equation}\label{denote}
S_\chi=\sum_{k=1}^{f} \overline{\chi}(k)\log (1+\zeta_{f}^{-k}),
\end{equation}
where $k$ running through a reduced system of residues modulo $f.$
From (\ref{ex-re}), the number $1+\zeta_{f}^{-k}$ (for $0<k<f$) can be represented as
\begin{equation}
\begin{aligned}
1+\zeta_{f}^{-k}=1+e^{-\frac{2\pi ki}{f}}=2\frac{(e^{\frac{\pi ki}{f}}+e^{-\frac{\pi ki}{f}})}{2}e^{-\frac{\pi ki}{f}},
\end{aligned}
\end{equation}
which is equivalent to the relation
\begin{equation}\label{z-ab}
\begin{aligned}
1+\zeta_{f}^{-k}=2\cos\left(\frac{\pi k}{f}\right)\left(\cos\left(-\frac{\pi k}{f}\right)+i\sin\left(-\frac{\pi k}{f}\right)\right).
\end{aligned}
\end{equation}
Therefore
\begin{equation}\label{ze-k}
\log (1+\zeta_{f}^{-k})=\begin{cases}
\log |1+\zeta_{f}^{-k}|-i\left(\frac{\pi k}{f}\right),& 0< k < \frac{f}{2},\\
\log |1+\zeta_{f}^{-k}|+i\left(\pi-\frac{\pi k}{f}\right),& \frac{f}{2}< k < f.
\end{cases}
\end{equation}
Further, since $1+\zeta_{f}^{-k}$ and $1+\zeta_{f}^{k}$ are conjugate,  we have
\begin{equation}\label{ze+k}
\log (1+\zeta_{f}^{k})=\begin{cases}
\log |1+\zeta_{f}^{k}|+i\left(\frac{\pi k}{f}\right),& 0< k < \frac{f}{2},\\
\log |1+\zeta_{f}^{k}|+i\left(\frac{\pi k}{f}-\pi\right),& \frac{f}{2} < k < f.
\end{cases}
\end{equation}

 Now assume that the character $\chi$ (and hence also $\overline{\chi}$) is even. Interchanging $k$ and $-k$ in (\ref{denote}), we have
\begin{equation}\label{LiAdd6}
S_\chi=\sum_{k=1}^{f} \overline{\chi}(-k)\log (1+\zeta_{f}^{k})=\sum_{k=1}^{f} \overline{\chi}(k)\log (1+\zeta_{f}^{k}).
\end{equation}
Combining  (\ref{denote}), (\ref{ze-k}) (\ref{ze+k}) and \eqref{LiAdd6}, this yields
\begin{equation}
\begin{aligned}
2S_\chi&=\sum_{k=1}^{f} \overline{\chi}(k)\left[\log(1+\zeta_{f}^{-k})+\log(1+\zeta_{f}^{k})\right] \\
&=2\sum_{k=1}^{f} \overline{\chi}(k)\log |1+\zeta_{f}^k| \\
&=2\sum_{k=1}^{f} \overline{\chi}(k)\log \left|2\cos\left(\frac{\pi k}{f}\right)\right|.
\end{aligned}
\end{equation}
Thus
\begin{equation}\label{Liadd1}
\begin{aligned}
S_\chi&=\sum_{k=1}^{f} \overline{\chi}(k)\log |1+\zeta_{f}^k| \\
&=\sum_{k=1}^{f} \overline{\chi}(k)\log \left|\cos\left(\frac{\pi k}{f}\right)\right|.
\end{aligned}
\end{equation}
since $\sum_{k=1}^{f} \overline{\chi}(k)=0.$
The
summation item in \eqref{Liadd1} is unchanged if we replace $k$ by $f-k$. Therefore,
\begin{equation}
\begin{aligned}
S_\chi&=2\sum_{1\leq k<f/2} \overline{\chi}(k)\log \left|\cos\left(\frac{\pi k}{f}\right)\right|.
\end{aligned}
\end{equation}
Then combing with 
(\ref{sim Le=1}) and (\ref{denote}),
we have\begin{equation}\begin{aligned}
L_{E}(1,\chi)&=\frac{2\tau(\chi)}{f}\sum_{1\leq k<f/2} \overline{\chi}(k)\log\left|\cos\left(\frac{\pi k}{f}\right)\right|\\
&=\frac{\tau(\chi)}f \sum_{k=1}^{f} \overline{\chi}(k)\log |1+\zeta_{f}^k|,
\end{aligned}\end{equation}
which is the desired result if $\chi$ is a even character.

If the character $\chi$ is odd, then  interchanging  $k$ and $-k$ in (\ref{denote}), we have
\begin{equation}
S_\chi=-\sum_{k=1}^{f} \overline{\chi}(k)\log(1+\zeta_{f}^{k}),
\end{equation}
and by (\ref{ze-k}) and (\ref{ze+k}), we have
\begin{equation}
\begin{aligned}
2S_\chi&=\sum_{k=1}^{f} \overline{\chi}(k)\left[\log(1+\zeta_{f}^{-k})-\log(1+\zeta_{f}^{k})\right] \\
&=-2\left[\sum_{1\leq k < f/2}\overline{\chi}(k)i\frac{\pi k}{f}+\sum_{f/2<k<f}\overline{\chi}(k)i\left(\frac{\pi k}{f}-\pi\right)\right]
\end{aligned}
\end{equation}
and
\begin{equation}\begin{aligned}
S_\chi&=-\frac{\pi i}{f}\left[\sum_{1\leq k < f/2}\overline{\chi}(k)k+\sum_{f/2<k<f}\overline{\chi}(k)\left(k-f\right)\right]\\
&=-\frac{\pi i}{f}\left[\sum_{1\leq k < f/2}\overline{\chi}(k)k+\sum_{-f/2<k<0}\overline{\chi}(k+f)k\right]\\
&=-\frac{2\pi i}{f}\sum_{1\leq k < f/2}\overline{\chi}(k)k.
\end{aligned}\end{equation}
The last equality is got by changing $k$ to $-k$ in the second summation and applying the oddness of $\chi$.
Then combing with 
(\ref{sim Le=1}) and (\ref{denote}),
we have  $$L_{E}(1,\chi)=-\frac{2\pi i \tau(\chi)}{f^2}\sum_{1\leq k < f/2}\overline{\chi}(k)k,$$
which is the desired result if $\chi$ is an odd character.

\section{Proof of Theorem \ref{main2}}
For $K^{+}=\mathbb{Q}(\mu_{m})^{+},$ we have $[K^{+}:\mathbb{Q}]=\frac{\varphi(m)}{2}$, where $\varphi$ is Euler-phi function.
Let $d$ be the absolute value of the discriminant of $K^{+}$ and $G=\mathbb{Z}(m)^{*}/\{\pm1\}$ is the Galois group of $K^{+}$ over $\mathbb{Q}$. By the Conductor-Discriminant-formula (\cite[p.28 Theorem 3.1]{Washington}),
\begin{equation}\label{LiAdd2}
d=\prod_{\chi}f_{\chi}\end{equation}
where $\chi$ runs over all the characters of $G$. 
From the functional equation of Dirichlet L-function, one can deduce that(\cite[p.36 Corollary 4.6]{Washington})
 \begin{equation}\label{LiAdd3}
\prod_{\chi}\tau(\chi)=d^{1/2}.
\end{equation}

For Re$(s)>1$,
\begin{equation}
\begin{aligned}
(1-\chi(2)2^{1-s})L(s,\chi)=\sum_{n=1}^{\infty}\frac{\chi(n)}{n^{s}}-2\sum_{n=1}^{\infty}\frac{\chi(2n)}{(2n)^{s}}
=L_{E}(s,\chi).
\end{aligned}
\end{equation}
By analytic continuation, the above equality holds in the whole complex plane. Thus
\begin{equation}\label{D-E-zeta}
(1-\chi(2))L(1,\chi)= L_E(1,\chi).
\end{equation}

Multiplying the class number formula (\ref{CNF}) by $\eta=\prod_{\chi\neq1}(1-\chi(2))$ 
and using \eqref{D-E-zeta}
 we have
\begin{equation}
\eta 2^{\frac{m}{2}-1}h^{+}R^{+}=d^{1/2}\prod_{\chi\neq 1}L_{E}(1,\chi)
\end{equation}
where the product over $\chi\neq 1$ is taken over non-trivial characters of $G$, or equivalently,  non-trivial even characters of $\mathbb{Z}(m)^{*}$. Applying theorem \ref{main1},
\begin{equation}\label{LiAdd4}
\eta 2^{\frac{m}{2}-1}h^{+}R^{+}=d^{1/2}\prod_{\chi\neq 1}\frac{\tau(\chi)}{f_{\chi}}\sum_{k=1}^{f}\chi(k)\log |1+\zeta_{f_{\chi}}^{k}|.
\end{equation}
Substituting \eqref{LiAdd2} and \eqref{LiAdd3} into \eqref{LiAdd4}, we get
\begin{equation}
\eta 2^{\frac{m}{2}-1}h^{+}R^{+}=\prod_{\chi\neq 1}\sum_{k=1}^{f}\chi(k)\log |1+\zeta_{f_{\chi}}^{k}|,
\end{equation}

Since
$$\sum_{k=1}^{f}\chi(k)\log |1+\zeta_{f_{\chi}}^{k}|=2\sum_{k\in G_{\chi}}\chi(k)\log |1+\zeta_{f_{\chi}}^{k}|$$ and there are exactly $\frac{\varphi(m)}{2}$ even characters and  $\frac{\varphi(m)}{2}-1$ non-trivial even characters,
we have \begin{equation}
\eta h^{+}R^{+}=\prod_{\chi\neq 1}\sum_{k\in G_{\chi}}\chi(k)\log |1+\zeta_{f_{\chi}}^{k}|,
\end{equation}
which is the desired result.

\section{Proof of Proposition \ref{Proposition1.7}}
For $m$ being a prime power, we know that $E=E^+W$, where $W=\langle-1, \zeta_{m}\rangle$ is the group of roots of unity in $K=\mathbb{Q}(\zeta_{m})$ (\cite[p.40, Theorem 4.12 and Corollary 4.13]{Washington}). Since $\tilde{\mathcal{E}}\supset W$, the following natural morphism
\begin{equation}
\begin{aligned}
\varphi: E^{+}&\rightarrow E/\tilde{\mathcal{E}}
\end{aligned}
\end{equation}
induced by the inclusion $E^{+}\rightarrow E$ is surjective. The kernel is $E^{+}\cap\tilde{\mathcal{E}}$, which by definition is just $\tilde{\mathcal{E}}^+$.
Therefore,
\begin{equation}
E/\tilde{\mathcal{E}}\cong E^{+}/\tilde{\mathcal{E}}^{+},\end{equation}
which is the desired result.

\section{Proof of Theorem \ref{Cor2}}
The Galois group of $\mathbb{Q}(\mu_{m})$ over $\mathbb{Q}$ is isomorphic to $\mathbb{Z}(m)^{*}$ under the map:
\begin{equation*}
 a\mapsto \sigma_{a}
 \end{equation*}
where~~ $\sigma_{a}: \zeta_{m}\mapsto\zeta_{m}^{a}$.
If $\epsilon_{1},\ldots,\epsilon_{r}$ is a basis for $E^{+}$ (mod roots of unity), then the regulator $R^{+}$ is the absolute value of the determinant
$$R(E)=R^{+}=\pm~\det\limits_{a,j}~ \log|\sigma_{a}\epsilon_{j}|,$$
where $j=1,\ldots,r$ and $a\in\mathbb{Z}(m)^{*}/\pm 1$ and $a\not\equiv \pm 1$ (mod $m$) (see \cite[p. 85]{Lang}).
Let $G=\mathbb{Z}(m)^{*}/\pm 1$ and view $a\in G$ and $a\neq 1$ in $G$. Using the cyclotomic units introduced in (\ref{cycunits}), we also form a new type of cyclotomic regulator as follows
\begin{equation}\label{Rcycle}R(\tilde{\mathcal{E}})=\widetilde{R_{\textrm{cyc}}}=\pm~\det\limits_{a,k}~ \log|\sigma_{a}\tilde{g}_{k}|\end{equation}
with $a, k\in G$.

As pointed out by Serre to Lang (see \cite[p. 90]{Lang}), the following determinant relation is due to Dedekind, February 1896, who communicated it to Frobenius in March.
\begin{theorem}[Dedekind determinant formula, {\cite[p.90, Theorem 6.2]{Lang}}]\label{DF}
Let $G$ be a finite group and $f$ be any (complex valued) function on $G$.
Then $$\det\limits_{a,b}f(ab^{-1})=\left[\sum_{a\in G}f(a)\right]\det\limits_{a,b\neq 1}[f(ab^{-1})-f(a)].$$
Therefore, for a finite abelian group $G$,
$$\prod_{\chi\neq 1}\sum_{a\in G}\chi(a)f(a^{-1})=\det\limits_{a,b\neq 1}[f(ab^{-1})-f(a)].$$
\end{theorem}
The above Dedekind determinant relation implies the following lemma.
\begin{lemma}\label{Lemma1}
We have for $G=\mathbb{Z}(m)^{*}/\pm 1$,
\begin{equation}
\pm~\det\limits_{a,k}~ \log|\sigma_{a}\tilde{g}_{k}|=\prod_{\chi\neq 1}\sum_{k\in G}\chi(k)\log |1+\zeta_{m}^{k}|.
\end{equation}
\end{lemma}
\begin{lemma}\label{Lemma2}
Let $m=p^{n}$ be a prime power and $\chi\neq 1$ be a nontrivial character of $G=\mathbb{Z}(m)^{*}/\pm 1$. Denote $G_{\chi}=\mathbb{Z}(f_{\chi})^{*}/\pm 1$. Then, we have
\begin{equation}
\sum_{k\in G_{\chi}}\chi(k)\log |1+\zeta_{f_{\chi}}^{k}|=\pm\sum_{k\in G}\chi(k)\log |1+\zeta_{m}^{k}|.
\end{equation}
\end{lemma}
\begin{proof}
Let $f_{\chi}=p^{s}$. We only need to prove the case $s<n$. Write the residue classes in $\mathbb{Z}(p^{n})$ in the form
$$y=k+c p^{s},~~\textrm{with}~~0\leq c<p^{n-s},$$
and $k$ ranges over a fixed set of representatives for residue classes of $\mathbb{Z}(p^{s})^{*}$,
then we have
\begin{equation}\label{l31}
\begin{aligned}
\sum_{y\mod p^{n}} \chi(y)\log|1+\zeta_{m}^{y}|&=\sum_{k \mod p^{s}} \sum_{c=0}^{p^{n-s}-1}\chi(k)\log|1+\zeta_{p^{n}}^{k}(\zeta_{p^{n}}^{p^{s}})^{c}|
\\&=\sum_{k\mod p^{s}}\chi(k) \sum_{c=0}^{p^{n-s}-1}\log|1+\zeta_{p^{n}}^{k}(\zeta_{p^{n}}^{p^{s}})^{c}|\end{aligned}
\end{equation}
Since $$\prod_{\lambda^{p^{n-s}}=1}(X-\lambda Y)=X^{p^{n-s}}-Y^{p^{n-s}},$$ we have
\begin{equation*}
\begin{aligned} \prod_{c=0}^{p^{n-s}-1}(1+\zeta_{p^{n}}^{k}(\zeta_{p^{n}}^{p^{s}})^{c})&=\prod_{c=0}^{p^{n-s}-1}(1-(-\zeta_{p^{n}}^{k})
(\zeta_{p^{n}}^{p^{s}})^{c})\\&=1-(-\zeta_{p^{n}}^{k})^{p^{n-s}}\end{aligned}
\end{equation*}
which implies that
\begin{equation}\label{l32}
\prod_{c=0}^{p^{n-s}-1}(1+\zeta_{p^{n}}^{k}(\zeta_{p^{n}}^{p^{s}})^{c})=\left\{
  \begin{array}{ll}
    1+\zeta_{p^{s}}^{k}, & if\ p\neq 2; \\
    1-\zeta_{p^{s}}^{k}, & if\ p=2.
  \end{array}
\right.
\end{equation}
For odd $p$, by (\ref{l31}) and (\ref{l32}), we obtain
\begin{equation}
\begin{aligned}
\sum_{y\mod p^{n}} \chi(y)\log|1+\zeta_{m}^{y}|
 &=\sum_{k\mod p^{s}}\chi(k) \sum_{c=0}^{p^{n-s}-1}\log|1+\zeta_{p^{n}}^{k}(\zeta_{p^{n}}^{p^{s}})^{c}|\\
 &=\sum_{k\mod p^{s}}\chi(k)\log|1+\zeta_{p^{s}}^{k}|.
\end{aligned}
\end{equation}
For $p=2$, we have $s\geq 2$ since $\chi\neq1$. In this case, similarly, we get
\begin{equation*}
\begin{aligned}
\sum_{y\mod p^{n}} \chi(y)\log|1+\zeta_{m}^{y}|
 &=\sum_{k\mod p^{s}}\chi(k)\log|1-\zeta_{p^{s}}^{k}|\\
&=\sum_{k\mod p^{s}}\chi(k)\log|1+\zeta_{p^{s}}^{k+p^s/2}|\\
&=\sum_{k\mod p^{s}}\chi(k-p^s/2)\log|1+\zeta_{p^{s}}^{k}|\\
&=\chi(1+p^s/2)\sum_{k\mod p^{s}}\chi(k)\log|1+\zeta_{p^{s}}^{k}|\\
&=-\sum_{k\mod p^{s}}\chi(k)\log|1+\zeta_{p^{s}}^{k}|
\end{aligned}
\end{equation*}
Since $(1+p^s/2)^2\equiv1\pmod{p^s} $, $\chi(1+p^s/2)=\pm1$. But it can not equal to $1$. Otherwise, the conductor of $\chi$ would be $p^s/2$. This explains the last equality.

Summing up, we have
\begin{equation}\label{LiAdd5}
\begin{aligned}
\sum_{y\mod p^{n}} \chi(y)\log|1+\zeta_{m}^{y}|
 =\pm\sum_{k\mod p^{s}}\chi(k)\log|1+\zeta_{p^{s}}^{k}|.
\end{aligned}
\end{equation}
The minus sign occurs if and only if $p=2$ and $s<n$.
 Recall  \begin{equation}
\begin{aligned}
G=\mathbb{Z}(p^{n})^{*}/\pm 1~\textrm{and}~G_{\chi}=\mathbb{Z}(f_{\chi})^{*}/\pm 1=\mathbb{Z}(p^{s})^{*}/\pm 1.
\end{aligned}
\end{equation}
As $\chi$ is even, dividing both sides of (\ref{LiAdd5}) by 2, we get the desired formula.
\end{proof}

We have the following result related to $h^{+}$ and $\widetilde{R_{\textrm{cyc}}}/R^{+}$.
\begin{proposition}\label{main3}
 Let $K=\mathbb{Q}(\mu_{m})$, $K^{+}=\mathbb{Q}(\mu_{m})^{+}$ and  $h^{+}$ be the class number of $K^{+}$.  Assume $m=p^{n}$ is a prime power. Then
\begin{equation*}
|\eta| h^{+}=\widetilde{R_{\textrm{cyc}}}/R^{+}.
\end{equation*}
\end{proposition}
\begin{proof}By (\ref{Rcycle}), Lemmas \ref{Lemma1} and \ref{Lemma2} and Theorem \ref{main2}, we have
\begin{equation*}\begin{aligned}
\widetilde{R_{\textrm{cyc}}}&=\pm~\det\limits_{a,k}~ \log|\sigma_{a}\tilde{g}_{k}|
\\&=\pm\prod_{\chi\neq 1}\sum_{k\in G_{\chi}}\chi(k)\log |1+\zeta_{f_{\chi}}^{k}|
\\&=|\eta| R^{+}h^{+}.
\end{aligned}
\end{equation*}
\end{proof}

The subgroup $\tilde{\mathcal{E}}^{+}$ has finite index in the group $E^{+}$ if and only if $\widetilde{R_{\textrm{cyc}}}\neq0$. Moreover, if $\widetilde{R_{\textrm{cyc}}}\neq0$,   by \cite[p.41, Lemma 4.15]{Washington} we have
\begin{equation}\label{Regulator2}(E:\tilde{\mathcal{E}})=(E^{+}: \tilde{\mathcal{E}}^{+})=\widetilde{R_{\textrm{cyc}}}/R^{+}.\end{equation}
From Proposition \ref{main3}, $\widetilde{R_{\textrm{cyc}}}\neq0$ if and only if $\eta=\prod_{\chi\neq1}(1-\chi(2))\neq0$. Clearly, for $m=2^n$, $\eta=1$. Therefore,
we also have a new formula related to $h^{+}$ and the unit indexes $(E:\tilde{\mathcal{E}})$ and $(E^{+}:\tilde{\mathcal{E}}^{+})$.
\begin{corollary}\label{Cor1}
 Let $K=\mathbb{Q}(\mu_{m})$, $K^{+}=\mathbb{Q}(\mu_{m})^{+}$ and  $h^{+}$ be the class number of $K^{+}$.  Assume either $m=2^n$; or $m=p^{n}$ is an odd prime power and $\chi(2)\neq 1$ for any non-trivial character $\chi$ of $G=\mathbb{Z}(m)^{*}/\pm 1$. Then
\begin{equation*}
|\eta| h^{+}=(E:\tilde{\mathcal{E}})=(E^{+}: \tilde{\mathcal{E}}^{+}).
\end{equation*}
\end{corollary}
The following lemma concerns  the non-vanishing  of $\eta$ for the odd case.
\begin{lemma}\label{add2}For $m=p^{n}$ being an odd prime power, $\chi(2)\neq 1$ for any non-trivial character $\chi$ of $G=\mathbb{Z}(m)^{*}/\pm 1$ if and only if $-1,2$ generate the group $\mathbb{Z}(m)^{*}$.\end{lemma}
\begin{proof} $\chi(2)\neq 1$ for any non-trivial character $\chi$ of $G$ if and only if $\widehat{G/\langle\bar{2}\rangle}$, the character group of $G/\langle\bar{2}\rangle$ is trivial, where $\bar{2}$ is the image of $2$ in $G$, which is equivalent to $G/\langle\bar{2}\rangle$ is trivial. Since
$G/\langle\bar{2}\rangle\cong \mathbb{Z}(m)^{*}/\langle-1, 2\rangle$, we get the desired result.
\end{proof}
 Finally by Corollary \ref{Cor1} and Lemma \ref{add2}, we get Theorem \ref{Cor2}.


\end{document}